\documentclass[12pt, reqno]{amsart} 

\usepackage[utf8]{inputenc} 

\usepackage{geometry} 
\usepackage{bbm}
\usepackage{verbatim}
\usepackage{graphicx}
\usepackage{amsmath}	
\usepackage{amssymb}	
\usepackage{amsthm}
\usepackage{rotating}	
\usepackage{subfigure}	
\usepackage{url}		
\usepackage[justification=raggedright,
singlelinecheck=false]{}	
\usepackage{indentfirst}	
\usepackage[splitrule]{footmisc}
\usepackage{tikz}

\tikzstyle{connector} = [->]
\usetikzlibrary{calc}
\pgfdeclarelayer{edgelayer}
\pgfdeclarelayer{nodelayer}
\pgfsetlayers{edgelayer,nodelayer,main}

\usepackage{graphicx} 

\usepackage{booktabs} 
\usepackage{array} 
\usepackage{paralist} 
\usepackage{verbatim} 
\usepackage{subfig} 

\newtheorem{thm}{Theorem}[section]
\newtheorem{lem}[thm]{Lemma}
\newtheorem{remark}{Remark}
\newtheorem{remark*}{Remark}

\newtheorem{conj}{Conjecture}

\theoremstyle{definition}
\newtheorem{defi}{Definition}

\newtheorem{probl}{Problem}

\def\NN{{\mathbb N}}

\def\cH{H}
\def\cP{\mathcal{P}_\ell^{(r)}}
\def\fP{L_\ell^{(r)}}
\def\ffP{L^{(r)}}

\newcommand{\floor}[1]{\left\lfloor #1 \right\rfloor}
\newcommand{\ceil}[1]{\left\lceil #1 \right\rceil}
\newcommand{\bracks}[1]{\left( #1 \right)}

\newcommand{\ex}[1]{\operatorname{ex}(#1)}
\newcommand{\exr}[1]{\operatorname{ex}_r(#1)}
\newcommand{\extwo}[1]{\operatorname{ex}_2(#1)}
\newcommand{\Exr}[1]{\operatorname{Ex}_r(#1)}

\title{Tur\'an Numbers for Forests of Paths in Hypergraphs}

\begin{document}

\author[N. Bushaw]{Neal Bushaw}

\author[N. Kettle]{Nathan Kettle}

\address{Neal Bushaw\hfill\break
School of Mathematical and Statistical Sciences, Arizona State University, Tempe, AZ 85287, USA}
\email{bushaw@asu.edu}

\address{Nathan Kettle\hfill\break
Department of Pure Mathematics and Mathematical Statistics, University of Cambridge, Cambridge CB3 0WB, UK}
\email{n.kettle@dpmms.cam.ac.uk}

\begin{abstract}
The \emph{Tur\'an number} of an $r$-uniform hypergraph $H$ is the maximum number of edges in any $r$-graph on $n$ vertices which does not contain $H$ as a subgraph. Let $\mathcal{P}_{\ell}^{(r)}$ denote the family of $r$-uniform loose paths on $\ell$ edges, $\mathcal{F}(k,l)$ denote the family of hypergraphs consisting of $k$ disjoint paths from $\mathcal{P}_{\ell}^{(r)}$, and $\fP$ denote an $r$-uniform linear path on $\ell$ edges. We determine precisely $\exr{n;\mathcal{F}(k,l)}$ and $\exr{n;k\cdot\fP}$, as well as the Tur\'an numbers for forests of paths of differing lengths (whether these paths are loose or linear) when $n$ is appropriately large dependent on $k,l,r$, for $r\geq 3$. Our results build on recent results of F\"uredi, Jiang, and Seiver who determined the extremal numbers for individual paths, and provide more hypergraphs whose Turan numbers are exactly determined.
\end{abstract}
\maketitle
\section{Introduction and Background}
Extremal graph theory is that area of combinatorics which is concerned with finding the largest, smallest, or otherwise optimal structures with a given property. Often, the area is concerned with finding the largest (hyper)graph avoiding some subgraph. We build on earlier work of F\"uredi, Jiang, and Seiver \cite{FurJiaSei}, who determined the extremal numbers when the forbidden hypergraph is a single linear path or a single loose path. In this paper, we determine precisely the exact Tur\'an numbers when the forbidden hypergraph is a forest of loose paths, or a forest of linear paths; our main results appear in Section \ref{sec:main}. This is one of only a few papers which gives exact Tur\'an numbers for an infinite family of hypergraphs; in this case, several such families. 

The \emph{Tur\'an number}, or \emph{extremal number}, of an $r$-uniform hypergraph $F$ is the maximum number of edges in any $r$-graph $H$ on $n$ vertices which does not contain $F$ as a subgraph. This is a natural generalization of the classical Tur\'an number for 2-graphs; we restrict ourselves to the case of $r$-uniform hypergraphs, as allowing the extremal number to count edges of different sizes obscures the true extremal structure.

Throughout, we use standard terminology and notation (see, e.g., \cite{MGT}). A \emph{hypergraph} is a pair $H\,=\,(V,E)$ consisting of a set $V$ of vertices and a set $E\subseteq\mathcal{P}(V)$ of edges. If $E\subseteq\binom{V}{r}$, then $H$ is an $r$-\emph{uniform hypergraph}; in this paper, we will restrict ourselves to this setting. If $|V|\,=\,n$, we will assume without loss of generality that $V\,=\,[n]\,=\,\{1,2,\ldots,n\}$. For two hypergraphs $G$ and $H$, we write $G\subseteq H$ if there is an injective homomorphism from $G$ into $H$. By disjoint, we will always mean \emph{vertex} disjoint; we use $G\cup H$ to denote the disjoint union of (hyper)graphs $G$ and $H$. Similarly, for $k\in\NN$, we use $k\cdot G$ to denote $k$ (vertex-)disjoint copies of $G$. We also make use of the indicator function;  $\mathbbm{1}_{E}=\begin{cases} 1 \textrm{ if $E$ holds,}\\0 \textrm{ else.}\end{cases}$.

Developing an understanding of the Tur\'an numbers and extremal graphs for a general $r$-graph $F$ is a classical and long-standing problem in extremal graph theory. The field began to take off in the 1940s, in the case of 2-graphs, when P\'al Tur\'an determined the extremal numbers for complete graphs of all orders; it is through this result that Tur\'an's name became synonymous with the field.

While this problem is well solved up to asymptotics when the forbidden graph has chromatic number at least three by the Erd\H{o}s-Stone Theorem \cite{ErdosStone}, things remain much murkier for bipartite graphs. We discuss this in Section \ref{sec:graphs}; first, we formally define the Tur\'an number for $r$-graphs as follows.

\begin{defi}
The \emph{$r$-uniform hypergraph Tur\'an Number}, or \emph{extremal number}, of a family $\mathcal{F}$ of $r$-uniform hypergraphs is defined as the following.
\[
 \exr{n;\mathcal{F}}\,=\,max\{|E(\cH)|:|V(\cH)|\,=\,n,\forall F\in\mathcal{F},\,F\not\subseteq H\}.
\]
\end{defi}

For a single hypergraph $F$, we will often write $\exr{n;F}$ for $\exr{n;\{F\}}$. We modify this definition slightly for lists of hypergraphs. As opposed to the above definition for a family of hypergraphs, where any member of the family is forbidden, here we are forbidding disjoint copies of all graphs in the list from appearing simultaneously.

\begin{defi}
The \emph{$r$-uniform hypergraph Tur\'an Number} of a list of $r$-uniform hypergraphs $F_1,F_2,\ldots,F_k$ is defined as the following.
 \[\exr{n;F_1,F_2,\ldots,F_k}\,=\,\exr{n;F_1\cup F_2\cup\ldots\cup F_k}.\]
\end{defi}

A hypergraph $H$ is called \emph{extremal for $F$} if $H$ is $F$-free and $|E(H)|\,=\,\exr{n;F}$; we denote by $\Exr{n;F}$ the family of $n$ vertex graphs which are extremal for $F$; similarly, $H$ is extremal for the list $F_1,F_2,\ldots,F_k$ if it does not contain disjoint copies of all graphs in the list and $|E(H)|\,=\,\exr{n;F_1,F_2,\ldots,F_k}$.

\subsection{Background for Graphs}\label{sec:graphs}
Before discussing paths in hypergraphs, on which this paper focuses, we discuss briefly the related results for paths in graphs of which the results in Section \ref{sec:main} are generalizations. We'll make use of the standard notation $\ex{n,F}=\extwo{n,F}$.

In 1959, Erd\H{o}s and Gallai proved the following result giving the extremal numbers for paths of a given length \cite{EG}. We note that the bound in Theorem~\ref{erdosgallai} is attained by taking disjoint copies of $K_{\ell-1}$. This extremal construction is unique as long as $n$ is divisible by $\ell-1$, and gives a tight bound in this case.

\begin{thm}[Erd\H{o}s-Gallai, 1959]
\label{erdosgallai}
For any $n,\ell\in\NN$, 
\[\ex{n;P_\ell}\leq\left(\frac{\ell-2}{2}\right)n.\]
\end{thm}

We note that a path can be viewed as an extreme kind of tree; it is the tree with largest diameter for its number of vertices. The opposite extreme is the star; here the diameter is only two. Forbidding the star $S_\ell$ with $\ell$ leaves is, in fact, simply imposing a maximum degree condition, and so $\ex{n;S_\ell}\leq\left(\frac{\ell-2}{2}\right)n$. This bound is tight, with the extremal graphs being all $(\ell-2)$-regular graphs. For general trees, this result is notoriously difficult and is known as the  Erd\H{o}s-S\'os Conjecture \cite{ErdosSos}.

\begin{conj}[Erd\H{o}s-S\'os, 1963]
\label{erdossosconj}
For any tree $T$ on $\ell$ vertices, $\ex{n;T}\,\leq\,\left(\frac{\ell-2}{2}\right)n$.
\end{conj}

In 2008, a proof of this conjecture was announced for very large trees by Ajtai, Koml\'os, Simonovits, and Szemer\'edi; this will appear in a series of upcoming papers \cite{AKSS1, AKSS3, AKSS2}. A sketch of this result can be found in a recent survey of F\"uredi and Simonovits \cite{FSSurvey}. For small trees, however, the conjecture is largely open. For a survey of other Tur\'an results for connected bipartite graphs, see, e.g., \cite{AKS, EGT, BushawThesis, BushawKettle}.

In \cite{BushawKettle}, the present authors determined precisely the extremal numbers for forests where each component has the same number of vertices in each bipartite class (such a forest is called \emph{equibipartite}), assuming the Erd\H{o}s-S\'os Conjecture holds for those trees in the forest, as well as the extremal numbers for forests of paths of the same odd length. Our main results in this paper give the hypergraph versions of these theorems for forests of (linear or loose) paths. For comparison, these graph results are included below. 

\begin{thm}[B.-K, 2011]
\label{thm:p3thm}
For $n\geq 7k$, the following holds.
\[\ex{n;k\cdot P_3}\,=\,\binom{k-1}{2}+(n-k+1)(k-1)+\floor{\frac{n-k+1}{2}}\]
\end{thm}

\begin{thm}[B.-K., 2011]
\label{thm:longpath}
For $k\geq 2$, $\ell\geq 4$, and $n\geq 2\ell+2k\ell\bracks{\ceil{\frac{\ell}{2}}+1}\binom{\ell}{\floor{\frac{\ell}{2}}}$, the following holds.
\[ \ex{n;k\cdot P_\ell}\,=\,\binom{k\floor{\frac{\ell}{2}}-1}{2}+\bracks{k\floor{\frac{\ell}{2}}
-1}\bracks{n-k\floor{\frac{\ell}{2}}+1}+\mathbbm{1}_{\{\ell\textrm{ is odd}\}}.\]
\end{thm}

\begin{thm}[B.-K., 2011]
\label{thm:treethm}
Let $F$ be an equibipartite forest on $2\ell$ vertices which is comprised of at least two trees. If the Erd\H{o}s-S\'os Conjecture holds for each component tree in $F$, then for $n\geq3\ell^2+32\ell^5\binom{2\ell}{\ell}$,
\[\ex{n;F}\,=\,\begin{cases} \binom{\ell-1}{2}+(\ell-1)(n-\ell+1) \textrm{, if }H\textrm{ admits a perfect matching}\\ (\ell-1)(n-\ell+1) \textrm{ otherwise.}\end{cases}\]
\end{thm}

With a generalization of these theorems in mind, we now proceed to a discussion of the hypergraph Tur\'an problem.

\subsection{Background for Hypergraphs}
In general, Tur\'an theory for $r$-uniform hypergraphs with $r\geq 3$ is much less developed than the theory for $2$-graphs. In the same paper in which Tur\'an proved his fundamental theorem on the extremal numbers for complete graphs \cite{Turan1}, he posed the natural question of determining $\exr{n;K_t^{(r)}}$, where $K_t^{(r)}$ denotes the complete $r$-uniform graph on $t$ vertices. Surprisingly, this problem remains open in all cases for $r>2$, even up to asymptotics. 

Determining extremal numbers precisely for hypergraphs is difficult indeed. Those results that do exist tend to be asymptotics, and exact results, with a few exceptions discussed below, are virtually always for small graphs on a few vertices. Such exact results exist (for large $n$) for the Fano plane, 4-books with 2, 3, or 4 pages, and a few other similarly small objects (see, e.g., \cite{Keevash} for a survey of hypergraph Tur\'an results). As a sample of a typical exact theorem in this area, we state the Erd\H{o}s-Ko-Rado Theorem below; this is perhaps the classical extremal result for hypergraphs \cite{EKR}.

\begin{thm}[Erd\H{o}s-Ko-Rado, 1961]
If $\cH$ is an $r$-uniform hypergraph on $n\geq 2r$ vertices in which every pair of edges intersects, then $|E(\cH)|\leq\binom{n-1}{r-1}$. That is, if we let $M_2^{(r)}$ denote the $r$-graph consisting of 2 disjoint edges, then for $n\geq2r$,
\[\exr{n;M_2^{(r)}}\,=\,\binom{n-1}{r-1}.\]
\end{thm}

We now continue to the main objects of focus in this paper: paths in hypergraphs. One can think of a matching in hypergraphs as being a forest of paths of length one. Thus perhaps starting point for this history is the following conjecture of Erd\H{o}s \cite{Erdos65}. If one considers $r$-uniform hypergraphs forbidding an $s+1$-matching, there are two natural constructions. Either one can take $r(s+1)-1$ vertices and all edges, or one can take a set of $s$ vertices and all edges intersecting this set; thus the Erd\H{o}s conjecture is as follows.

\begin{conj}[Erd\H{o}s, 1965]
\label{conj:Erdos}
Let $M_{s+1}^{(r)}$ denote ${s+1}$ disjoint $r$-edges, $A_r=\binom{[r(s+1)-1]}{r}$, and $B_r(n)=\left\{F\in\binom{[n]}{r}:F\cap[s]\neq\emptyset\right\}$. Then
\[\exr{n;M_s^{(r)}}=\max{\{|A_r|,|B_r(n)|\}}.\]
\end{conj}

In the same paper as his conjecture, Erd\H{o}s proved this result for $n>n_0(r,s)$; it is this result to which the main theorems of this paper correspond. In fact, both Theorem \ref{thm:multlinearsame} and Theorem \ref{thm:multloosesame}, with $\ell=1$, prove the above conjecture for large $n$. The bound on $n_0(r,s)$ has been gradually improved over the years: Bollob\'{a}s, Daykin, and Erd\H{o}s \cite{BDE} proved $n_0(r,s)\leq2r^3s$; this was later improved by Huang, Loh, and Sudakov \cite{HLS} to $n_0(r,s)\leq3r^2s$. In the case of $r=3$, there has been significant recent progress: Frankl, R\"{o}dl, and Rucinski \cite{FRR} showed $n_0(3,s)\leq 4s$, {\L}uczak and Mieczkowska \cite{LM} proved the conjecture for $r=3$ and $s>s_0$, and finally Frankl \cite{Fr} settled this case for all $s$ and $n$.

We now move to paths of longer length in hypergraphs, noting first that there are several natural generalizations of paths in graphs to paths in hypergraphs. Thus we give three different definitions of paths in hypergraphs; we present these from most general to most specific.

\begin{defi}
A \emph{Berge path} of length $\ell$ in a hypergraph $\cH$ is a family of distinct edges $\{F_1,\ldots,F_\ell\}~\subseteq~E(\cH)$ along with a family of vertices $\{v_1,\ldots,v_{\ell+1}\}\subseteq V(\cH)$ such that for each $i\in\{1,\ldots,\ell\}$, $v_i,v_{i+1}\in F_i$.
\end{defi}

\begin{defi}
A \emph{loose path} of length $\ell$ in a hypergraph $\cH$ is a family of distinct edges $\{F_1,\ldots,F_\ell\}\subseteq E(\cH)$ such that $F_i\cap F_j\neq\emptyset$ iff $|i-j|\,=\,1$. We use $\cP$ to denote the family of $r$-uniform loose paths on $\ell$ edges.
\end{defi}

\begin{defi}
A \emph{linear path} of length $\ell$ in a hypergraph $\cH$ is a family of distinct edges $\{F_1,\ldots,F_\ell\}\subseteq E(\cH)$ such that $\left\lvert F_i\cap F_j\right\rvert\,=\,1$ if $\left\lvert i-j\right\rvert\,=\,1$, and $ F_i\cap F_j\,=\,\emptyset$ otherwise. We use $\fP$ to denote an $r$-uniform linear path on $\ell$ edges. 
\end{defi}

We provide examples of Berge, loose, and linear 4-paths in Figures \ref{fig:bergepath}, \ref{fig:loosepath}, and \ref{fig:linearpath}, respectively.

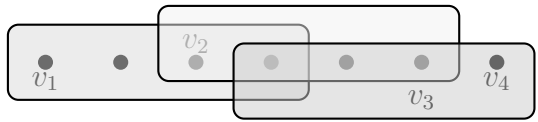
\begin{figure}
\centering
\begin{tikzpicture}[thick]
    \begin{pgfonlayer}{nodelayer}
	\foreach \x in {0,...,6}{\node (t_\x) at ($(\x,0)$){};}
    \foreach \x in {0,...,6} {\fill (t_\x) circle (0.1);} 
	\node (v_1) at ($(0,-0.25)$){$v_1$};
	\node (v_2) at ($(2,0.25)$){$v_2$};
	\node (v_3) at ($(5,-0.5)$){$v_3$};
	\node (v_4) at ($(6,-0.25)$){$v_4$};
  \end{pgfonlayer}
  \begin{pgfonlayer}{edgelayer}
  \end{pgfonlayer}
    
\begin{scope}[fill opacity=0.5, draw opacity = 1.0]
     \draw [fill=lightgray!60, rounded corners, thick] (-.5,-.5) rectangle (3.5,0.5);
     \draw [fill=lightgray!20, rounded corners, thick] (1.5,-.25) rectangle (5.5, 0.75);
     \draw [fill=lightgray!80, rounded corners, thick] (2.5, -.75) rectangle (6.5, 0.25);
     \end{scope}
\end{tikzpicture}

\caption{A 4-uniform Berge path on 3 edges.}
\label{fig:bergepath}
\end{figure}

\begin{figure}
\centering
\begin{tikzpicture}[thick]
    \begin{pgfonlayer}{nodelayer}
	\foreach \x in {0,...,8}{\node (t_\x) at ($(\x,0)$){};}
	\foreach \x in {0,...,7}{\node (b_\x) at ($(\x,-2)$){};}
    \foreach \x in {0,...,8} {\fill (t_\x) circle (0.1);}
    \foreach \x in {0,...,7}{\fill (b_\x) circle (0.1);}
  \end{pgfonlayer}
  \begin{pgfonlayer}{edgelayer}
  \end{pgfonlayer}
    
\begin{scope}[fill opacity=0.5, draw opacity = 1.0]
     \draw [fill=lightgray!80, rounded corners, thick] (-.5,-.5) rectangle (3.5,0.5);
     \draw [fill=lightgray!40, rounded corners, thick] (1.5,-.25) rectangle (5.5, 0.75);
     \draw [fill=lightgray!80, rounded corners, thick] (4.5, -.5) rectangle (8.5, 0.5);
     \draw [fill=lightgray!80, rounded corners, thick] (-.5,-2.25) rectangle (3.5,-1.5);
     \draw [fill=lightgray!80, rounded corners, thick] (1.5,-2.75) rectangle (5.5,-1.75);
     \draw [fill=lightgray!80, rounded corners, thick] (3.75,-2.5) rectangle (7.5,-1.5);
     \end{scope}
\end{tikzpicture}
\caption{Two 4-uniform loose paths, each on 3 edges.}
\label{fig:loosepath}
\end{figure}
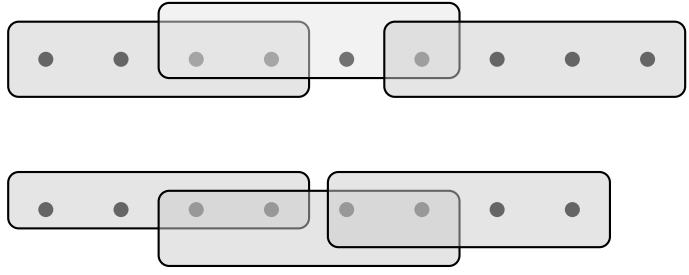

\begin{figure}
\centering
\begin{tikzpicture}[thick]
    \begin{pgfonlayer}{nodelayer}
	\foreach \x in {0,...,9}{\node (t_\x) at ($(\x,0)$){};}
    \foreach \x in {0,...,9} {\fill (t_\x) circle (0.1);}
  \end{pgfonlayer}
  \begin{pgfonlayer}{edgelayer}
  \end{pgfonlayer}
    
\begin{scope}[fill opacity=0.5, draw opacity = 1.0]
     \draw [fill=lightgray!80, rounded corners, thick] (-.5,-.5) rectangle (3.5,0.5);
     \draw [fill=lightgray!40, rounded corners, thick] (2.5,-.25) rectangle (6.5, 0.75);
     \draw [fill=lightgray!80, rounded corners, thick] (5.5, -.5) rectangle (9.5, 0.5);
     \end{scope}
\end{tikzpicture}

\caption{A 4-uniform linear path on 3 edges.}
\label{fig:linearpath}
\end{figure}
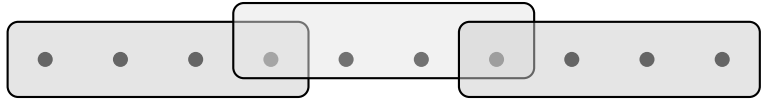

The following results were proven by F\"uredi, Jiang, and Seiver \cite{FurJiaSei}. These results are quite significant, as mentioned in the introduction, as they provide exact Tur\'an numbers for an infinite family of hypergraphs. In this sense, Theorems~\ref{thm:oneloosepath} and \ref{thm:onelinearpath} are the first of their kind.

\begin{thm}[F\"uredi-Jiang-Seiver, 2011]
\label{thm:oneloosepath}
Let $r\geq3$, $\ell\geq 1$. Letting $t\,=\,\left\lfloor\frac{\ell+1}{2}\right\rfloor-1$, $c_\ell\,=\,\mathbbm{1}_{\{\ell\textrm{ is even}\}}$, and $n$ sufficiently large,
\[
 \exr{n;\cP}\,=\,\binom{n-1}{r-1}+\binom{n-2}{r-1}+\ldots+\binom{n-t}{r-1}+c_\ell.
\]
\end{thm}

The unique extremal family consists of all the $r$-sets in $[n]$ which meet some fixed set $S$ of $t$ vertices, plus one additional $r$-set disjoint from $S$ when $\ell$ is even.

\begin{thm}[F\"uredi-Jiang-Seiver, 2011; Kostochka-Mubayi-Verstra\"ete, 2013]
\label{thm:onelinearpath}
Let $r\geq 3$, $\ell\geq 1$. Letting $t\,=\,\left\lfloor\frac{\ell+1}{2}\right\rfloor-1$, $d_\ell\,=\,0$ if $\ell$ is odd and $d_\ell\,=\,\binom{n-t-2}{r-2}$ if $\ell$ is even, and $n$ sufficiently large,
\[
 \exr{n;\fP}\,=\,\binom{n-1}{r-1}+\binom{n-2}{r-1}+\ldots+\binom{n-t}{r-1}+d_\ell.
\]
\end{thm}

For $\ell$ odd, the unique extremal family consists of all the $r$-sets in $[n]$ which meet some fixed set $S$ of $t$ vertices. For $\ell$ even, we have these edges plus all the $r$-sets in $[n]\setminus S$ containing some two fixed elements not in $S$.

In Theorem \ref{thm:onelinearpath}, the case $r\geq4$ was proved by F\"uredi, Jiang, and Seiver \cite{FurJiaSei}, and the case $r=3$ by Kostochka, Mubayi, and Verstra\"ete \cite{kostochka2013turan}. 

In addition to these results for paths, we will need the following result due to Keevash, Mubayi, and Wilson (Theorem~1.3 in \cite{keemubwil}).

\begin{thm}[Keevash-Mubayi-Wilson, 2006]
\label{thm:KeeMubWil}
Let $\cH$ be an $r$-uniform hypergraph on $n$ vertices with no singleton intersection, where $r\geq3$. Then
\[
 \left\lvert E\left(\cH\right)\right\rvert\leq\binom{n}{r-2}.
\]
\end{thm}

This gives us an upper bound on the number of edges in a hypergraph where no two edges intersect at \emph{exactly} one vertex. This will be useful to us in the case of linear paths, as the intersection of consecutive edges in these paths is precisely a single vertex.

\section{Main Results}\label{sec:main}
As mentioned in the introduction, our main results give exact Tur\'an numbers for forests of hyperpaths. We have several theorems here, dealing with forests of linear paths and loose paths, and with all paths having the same length or differing lengths. As every linear $r$-path on $\ell$ edges is isomorphic, the results for these types of paths avoid some notational difficulties associated with loose paths. Thus we state these results first, although our proofs will appear in the reverse order, as loose paths avoid the difficulties associated with singleton intersections, and thus avoid the use of Theorem~\ref{thm:KeeMubWil} above.

\begin{thm}
\label{thm:multlinearsame}
Let $r\geq 3$, $\ell\geq 1$, $k\geq 2$, and $n$ sufficiently large. Then letting $t\,=\,k\left\lfloor\frac{\ell+1}{2}\right\rfloor-1$, $d_\ell\,=\,0$ if $\ell$ is odd, and $d_\ell\,=\,\binom{n-t-2}{r-2}$ if $\ell$ is even,
\[\exr{n;k\cdot\fP}\,=\,\binom{n-1}{r-1}+\ldots+\binom{n-t}{r-1}+d_\ell.\]
\end{thm}

\begin{thm}\label{thm:multlineardiff}
Let $r\geq 3$, $k\geq 2$, $\ell_1,\ldots,\ell_k\geq 1$, and $n$ sufficiently large. Setting $t\,=\,\sum_{i\in[k]}\left\lfloor\frac{\ell_i+1}{2}\right\rfloor-1$ and $d_l=0$ when at least one of the $\ell_i$ is odd and $d_l=\binom{n-t-2}{r-2}$ otherwise, the following holds.
\[
\exr{n;\ffP_{\ell_1},\ldots,\ffP_{\ell_k}}\,=\,\binom{n-1}{r-1}+\ldots+\binom{n-t}{r-1}+d_\ell.
\]
\end{thm}

We now continue to forests of loose paths. Here, we need some extra notation. Our results have the same flavor as the above; we give the maximum number of edges in a graph which does not have $k$ vertex disjoint loose paths of given lengths. However, not all loose paths of a given length are isomorphic. Thus we define the following family of graphs, where each member consists of $k$ disjoint paths.
\[\mathcal{F}(k,\ell)\,=\,\{P_1\cup\ldots\cup P_k:P_i\in\cP\mbox{ for each }i\in[k]\}.\]

We also define a similar but slightly more complicated family for paths of differing lengths.
\[\mathcal{F}^\prime(\ell_1,\ell_2,\ldots,\ell_k)\,=\,\{P_1\cup\ldots\cup P_k:P_i\in\mathcal{P}_{\ell_i}^{(r)}\mbox{ for each }i\in[k]\}.\]

With these definitions, we can state our remaining results as follows.

\begin{thm}
\label{thm:multloosesame}
Let $r\geq 3$, $\ell\geq 1$, $k\geq 1$, and $n$ sufficiently large. Then letting $t\,=\,k\left\lfloor\frac{\ell+1}{2}\right\rfloor-1$, $c_\ell\,=\,\mathbbm{1}_{\{\ell\textrm{ is even}\}}$, 
\[\exr{n;\mathcal{F}(k,l)}\,=\,\binom{n-1}{r-1}+\ldots+\binom{n-t}{r-1}+c_\ell.\]
\end{thm}

\begin{thm}
\label{thm:multloosediff}
Let $r\geq 3$, $k\geq 2$, $\ell_1,\ldots,\ell_k\geq 1$, and $n$ sufficiently large. Letting $t\,=\,\sum_{i\in[k]}\left\lfloor\frac{\ell_i+1}{2}\right\rfloor-1$ and $c_\ell\,=\,\mathbbm{1}_{\{\textrm{every $\ell_i$ is even}\}}$,

\[\exr{n;\mathcal{F}^\prime(\ell_1,\ldots \ell_k)}\,=\,\binom{n-1}{r-1}+\ldots+\binom{n-t}{r-1}+c_\ell.\]
\end{thm}

The proofs for each of these results follows a standard form. Given a hypergraph with many edges (i.e. more than the claimed extremal number), we assert that any individual path of the appropriate type found in our hypergraph must have many edges incident to its vertex set; otherwise, the hypergraph induced by the vertices outside this path contains the rest of the specified forest by induction. We then examine sets of vertices which are incident to many edges. From here, we determine that there is a set of vertices which is large enough that we can construct our entire forest using  these vertices as building blocks, in the case that a particular path is found in the rest of the graph.

This is enough to give us the structural result we need, and we can then simply count potential edges. This is complicated somewhat in the case of linear paths, as we need to find edges which intersect in only a single vertex; nevertheless, our overall scheme is the same. In order to avoid redundancy, we shall prove only Theorem \ref{thm:multlinearsame} and Theorem \ref{thm:multloosediff}. The differing lengths in Theorem \ref{thm:multlineardiff} can be dealt with precisely as in the proof of Theorem \ref{thm:multloosediff}, and Theorems \ref{thm:oneloosepath}, and \ref{thm:multloosediff} together give Theorem \ref{thm:multloosesame}.

\section{Forests of Loose Paths}\label{sec:multloose}
We note some ambiguity in the base case: instead of forbidding a particular loose path of length $\ell$, the family of loose paths of length $\ell$ is forbidden. Thus, as in Theorem~\ref{thm:oneloosepath} we are not able to guarantee the existence of a \emph{particular} loose path of a given length, only that \emph{some} $\ell$-edge loose path exists.

For ease of notation, we let $t\,=\,\sum_{i\in[k]}\left\lfloor\frac{\ell_i+1}{2}\right\rfloor-1$, $c_\ell\,=\,\mathbbm{1}_{\{\textrm{every $\ell_i$ is even}\}}$, and define
\[h(n,r,\left\{\ell_1,\ldots,\ell_k\right\})\,=\,\binom{n-1}{r-1}+\ldots+\binom{n-t}{r-1}+c_\ell.\]

We note that the hypergraph on $n$ vertices that has every edge incident to a specified set $S$ of $t$ vertices, along with a single edge disjoint from $S$ when all of the paths are even, has exactly $h(n,r,\left\{\ell_1,\ldots,\ell_k\right\})$ edges and does not contain a copy of any member of  $\mathcal{F}^\prime(\ell_1,\ell_2,\ldots,\ell_k).$

We start with a lemma that will be of use in the proofs of each of our main theorems.

\begin{lem}
\label{lem:pathmaking}
Let $c$ and $t$ be positive constants and $n$ large enough. Then each $r$-hypergraph $H$ on $n$ vertices with at least $cn^r$ edges and a specified set of vertices $T\subset V(H)$ of size $|T|\leq t$ contains a pair of 1-intersecting edges that is vertex disjoint from $T$.
\end{lem} 

\begin{proof}
The number of edges that meet $T$ is at most $t\binom{n-1}{r-1},$ so there are at least $cn^r-t\binom{n-1}{r-1}$ edges in the hypergraph $H$ restricted to vertex set $V(H)\setminus T$. As this is more than $\binom{n-t}{r-2}$ for $n$ sufficiently large, we can find a pair of 1-intersecting edges by Theorem \ref{thm:KeeMubWil}.

\end{proof}

\begin{proof}[Proof of Theorem~\ref{thm:multloosediff}]
The case $k=1$ is provided by Theorem \ref{thm:oneloosepath}. We proceed by induction on $k$; thus assume that $k\geq 2$, and that $\cH$ is a hypergraph on $n$ vertices with $|E(\cH)|\,=\,m>h(n,r,\left\{\ell_1,\ldots,\ell_k\right\})$. If any of the $\ell_i$ is even, we rearrange the list so that $\ell_1$ is even for convenience.

Since $h(n,r,\left\{\ell_1,\ldots,\ell_k\right\})>h(n,r,{\ell_1})$, for $n$ large enough, we can find at least one loose path on $\ell_1$ vertices inside $\cH$. Consider one of these $\ell_1$-paths, say on vertex set $P$. Certainly $|E(V(\cH)\setminus P)|\leq h(n-|P|,r,\left\{\ell_2,\ldots,\ell_k\right\})$, or else by induction, the graph on $V(\mathcal{H})\setminus P$ contains a member of $\mathcal{F}(\ell_2,\ldots,\ell_k)$; these alongside the loose path on $P$ form a member of $\mathcal{F}(\ell_1,\ldots,\ell_k)$.

Letting $n_P$ denote the number of edges of $\cH$ incident to vertices in $P$, we have that
\begin{align}
n_P&\geq m-h(n-|P|,r,\left\{\ell_2,\ldots,\ell_k\right\})\nonumber\\
&\geq h(n,r,\left\{\ell_2,\ldots,\ell_k\right\})-h(n-(\ell+1),r,\left\{\ell_2,\ldots,\ell_k\right\})\nonumber\\
&=\,\frac{\left\lfloor\frac{\ell_1+1}{2}\right\rfloor n^{r-1}}{(r-1)!}+O(n^{r-2}).\label{eqn:npupperMult}
\end{align}

We now focus on counting sets of vertices which can be used to easily `finish' edges started by vertices in $P$. With this in mind, for every set $R$ of $r-1$ vertices from $V(\cH)\setminus P$, we define
\[
A_{R}\,=\,\left\{E^\prime\in E\left(\cH\right): R\subseteq E^\prime\textrm{ and }E^\prime\setminus R\in P\right\}.
\]

We now break the $(r-1)$ subsets of $V(\cH)\setminus P$ into two sets, dependent on the size of their respective $A_{R}$:
\[A\,=\,\left\{R\in\left(V(\cH)\setminus P\right)^{(r-1)}: |A_{R}|\leq\left\lfloor\frac{\ell_1+1}{2}\right\rfloor-1\right\}\]
\[B\,=\,\left\{R\in\left(V(\cH)\setminus P\right)^{(r-1)}: |A_{R}|\geq\left\lfloor\frac{\ell_1+1}{2}\right\rfloor\right\}.\]

Counting edges entirely contained in $P$ and the edges incident to the sets $A$, $B$ defined above, we have that
\begin{align}
 n_P&\leq\binom{\left\lvert V(P)\right\rvert}{2}\binom{n}{r-2}+\left(\left\lfloor\frac{\ell_1+1}{2}\right\rfloor-1\right)\left\lvert A\right\rvert+\left\lvert V(P)\right\rvert\left|B\right|,\nonumber\\
&\leq\binom{\left\lvert V(P)\right\rvert}{2}\binom{n}{r-2}+\left(\left\lfloor\frac{\ell_1+1}{2}\right\rfloor-1\right)\binom{n}{r-1}+r\ell_1|B|\label{eqn:nplowerMult}.
\end{align}

By comparison of the upper and lower bounds on $n_P$, \eqref{eqn:npupperMult}, and \eqref{eqn:nplowerMult}, we have that
\begin{align}
 \left\lvert B\right\rvert\geq\frac{n^{r-1}}{r\ell_1\left(r-1\right)!}+O(n^{r-2}).\label{eqn:bsizeMult}
\end{align}

To each set \(R\in B\) we associate a set of \(\left\lfloor\frac{\ell_1+1}{2}\right\rfloor\) vertices from \(A_R\) arbitrarily. From \eqref{eqn:bsizeMult}, we see that some set of \(\left\lfloor\frac{\ell_1+1}{2}\right\rfloor\) vertices is chosen many times; here `many' is at least:
\begin{align}
\binom{\left\lvert V(P)\right\rvert}{\left\lfloor\frac{\ell_1+1}{2}\right\rfloor}^{-1}\frac{n^{r-1}}{\left(r-1\right)!r\ell_1}+O(n^{r-2})\geq\binom{r\ell_1}{\left\lfloor\frac{\ell_1+1}{2}\right\rfloor}^{-1}\frac{n^{r-1}}{\left(r-1\right)!r\ell_1}+O(n^{r-2}).\label{eqn:manyMult}
\end{align}

Thus each loose path on $\ell_1$ vertices in $\cH$ contains a subset of $\left\lfloor\frac{\ell_1+1}{2}\right\rfloor$ vertices that has many common edge-finishing $(r-1)$-sets in the rest of the graph.

Let \(U\) be such a set of $\left\lfloor\frac{\ell_1+1}{2}\right\rfloor=\ell^\prime$ vertices and $X_U$ be the set of common edge-finishing $(r-1)$-sets. Since \[\left\lvert E\left(V\left(\cH\right)\setminus U\right)\right\rvert>h(n-\ell^\prime,r,\{\ell_2,\ldots,\ell_k\}),\] we can find $k-1$ vertex disjoint loose paths of appropriate lengths on vertices inside $V(\cH)\setminus U$, say on vertex set $W$ with $|W|<(\ell_2 +\ldots+\ell_k)r$. 

We shall now find a loose path of with $\ell_1$ edges all of which are of the form $X^\prime \cup \{u\},$ where $X^\prime\in X_U$ and $u\in U.$ Applying Lemma \ref{lem:pathmaking} to the $(r-1)$-hypergraph on vertex set $V(\mathcal{H})\setminus U$ and edge set $X_U,$ with $c<\frac{1}{(r-1)!r\ell_1\binom{r\ell_1}{\ell^\prime}}$ and $t=(\ell_2+\ldots +\ell_k)r+\ell^\prime(2r-3),$ we see from Equation \ref{eqn:manyMult} that we can find a pair of 1-intersecting edges $Y_1$ and $Z_1$ disjoint from $W.$ We can repeat this argument $\ell^\prime$ more times to find vertex disjoint 1-intersecting edges $Y_2,Z_2,\ldots,Y_{\ell^\prime+1},Z_{\ell^\prime+1},$ all also disjoint from $W.$ Setting $U~=~\{u_1,u_2,\ldots, u_{\ell^\prime}\},$ we can make a loose path of length $\ell_1,$ with edges 
\[Z_1\cup\{u_1\},Y_2\cup\{u_1\},Z_2\cup\{u_2\},\ldots,Z_{\ell^\prime}\cup\{u_{\ell^\prime}\},Y_{\ell^\prime+1}\cup\{u_{\ell^\prime}\},\] where the last edge is only required if $\ell_1$ is odd. In fact this path is not just loose but also linear. Thus we have constructed $k$ disjoint loose paths and so our initial graph can not have more than $h(n,r,\{\ell_1,\ldots,\ell_k\})$ edges.
\end{proof}

\begin{remark}
As mentioned in the proof sketch, a major step is examining sets of vertices which are incident to many edges. From here, we determine that there is a set of vertices which is large enough that we can construct our entire forest using these vertices as building blocks, in the case that a particular path is found in the rest of the graph. When proving Theorem \ref{thm:multlineardiff}, this step could also be deduced from Theorem 6.2 of \cite{FurJiaSei} for $r\geq 4$; embedding the necessary linear forests can then be carried out as above.
\end{remark}

\section{Multiple Linear Paths}\label{sec:multlinear}
We now proceed to forests of linear paths. The techniques are similar to the proofs for loose paths, but as the intersection of edges in linear paths have a particular shape (i.e. just a single vertex), we require some extra tools. The difference arises in the last steps, where we are building a linear path out of common neighborhoods. Instead of simply taking any two intersecting edges, as in the case for loose paths, we need to find edges which intersect appropriately for building linear paths.

\begin{proof}[Proof of Theorem~\ref{thm:multlinearsame}]
For ease of notation, we define $a(n,r,k,\ell)\,=\,\binom{n-1}{r-1}+\ldots+\binom{n-t}{r-1}+d_\ell$.

The case $k=1$ is provided by Theorem \ref{thm:onelinearpath}. We proceed by induction on $k$. Let $k\geq 2$, and let $\cH$ be a hypergraph on $n$ vertices and with $|E(\cH)|\,=\,m>a(n,r,k,\ell)$. Since $a(n,r,k,\ell)>a(n,r,1,\ell)$, for $n$ large enough, we can find at least one linear path inside $\cH$.

As in Section \ref{sec:multloose}, consider one of these linear paths, say on vertex set $P$. Certainly $|E(V(\cH)\setminus P)|\leq a(n-|P|,r,k-1,\ell)$, or else by induction, the graph on $V(H)\setminus P$ contains $(k-1)\cdot\fP$; these along with the linear path on $P$ form $k\cdot\fP$. As before, we let $n_P$ denote the number of edges of $\cH$ incident to vertices in $P$; by identical calculations, we have that
\begin{align}
n_P\geq\,\frac{\left\lfloor\frac{\ell+1}{2}\right\rfloor n^{r-1}}{(r-1)!}+O(n^{r-2}).\label{eqn:npupperlin}
\end{align}

Again we focus on counting sets of vertices which can be used to build edges started by vertices in $P$, defining $A_R$, $A$, and $B$ identically to the proof for loose paths; by the same counting arguments, we get that
\begin{align}
 \left\lvert B\right\rvert\geq\frac{\frac{n^{r-1}}{\left(r-1\right)!}+O(n^{r-2})}{r\ell}.\label{eqn:bsizelin}
\end{align}

To each set \(R\in B\) we now associate a set of \(\left\lfloor\frac{\ell+1}{2}\right\rfloor\) vertices from \(A_R\) arbitrarily. From \eqref{eqn:bsizelin}, we see that some set $U$ of \(\left\lfloor\frac{\ell+1}{2}\right\rfloor\) vertices is chosen many times; here `many' is again at least:
\begin{align}
\binom{r\ell}{\left\lfloor\frac{\ell+1}{2}\right\rfloor}^{-1}\frac{n^{r-1}}{\left(r-1\right)!r\ell}+O(n^{r-2}).\label{eqn:manylin}
\end{align}

Thus, as before, each linear path found in $\cH$ has a set of $\left\lfloor\frac{\ell+1}{2}\right\rfloor$ vertices which have many common edge-finishing $(r-1)$-sets in the rest of the graph, and we can again find  $(k-1)\cdot\fP$ on vertices inside $V(\cH)\setminus U$, say on vertex set $W$.

We are now in the same position we were at the end of the proof of Theorem \ref{thm:multloosediff} and so we can construct a linear path with $\ell$ edges and thus we have constructed $k\cdot\fP$.

We note that the hypergraph on $n$ vertices in which each edge is incident to a specified set $S$ of $t$ vertices, along with all edges disjoint from $S$ containing some two fixed vertices not in $S$ when $k$ is even, gives a graph with exactly $\binom{n-1}{r-1}+\ldots+\binom{n-t}{r-1}+d_\ell$ edges and without $k$ vertex disjoint linear paths of length $\ell$; thus our result gives the exact value of the extremal function.
\end{proof}

Modifying the above proof in the same manner as was used to deal with differing path lengths in the proof of  Theorem~\ref{thm:multloosediff}, we obtain Theorem~\ref{thm:multlineardiff} for multiple linear paths of varying lengths; the lower bound is again given by the hypergraph on $n$ vertices in which each edge is incident to at least one of a specified set $S$ of $t$ vertices, along with all edges disjoint from $S$ containing some two fixed elements not in $S$ when each of the $\ell_i$ is even.

\section{Open Problems}

Using different methods, Füredi and Jiang \cite{FJ} (for $r\geq 5$) and subsequently Kostochka, Mubayi, and Verstra\"ete \cite{kostochka2013turan} (for all $r\geq 3$), found the extremal number of the linear cycle $C_{\ell}^{(r)}$, the hypergraph formed from $L_{\ell-1}^{(r)}$ and one other edge that shares exactly one vertex with each of the two end edges of $L_{\ell-1}^{(r)}$, for all $r$ and $l$. In particular for $r\geq 3$, $l\geq 4$, and $(r,l)\neq (3,4)$, the extremal number satisfies,

\[
\exr{n;C_{\ell}^{(r)}} = \exr{n;L_{\ell}^{(r)}},
\]
and the extremal hypergraphs are the same as well.

Using similar methods to those in this paper, the `eventual extremal number' (that is, the extremal number for large $n$) for hypergraphs consisting of linear cycles and linear paths can be determined. The main reason this is possible is because of the common structure shared by the extremal hypergraph for linear paths and linear cycles.

\begin{probl}What hypergraphs have an eventual extremal hypergraph structure similar to the linear path? That is, consist of every edge adjacent to a set of $t$ vertices, and $o(n^{r-1})$ other edges. \end{probl}

As well as paths, there are several notions of trees in hypergraphs as well. We define the tight $r$-tree inductively as follows. Every $r$-graph consisting of a single edge is an $r$-tree. Suppose that $\mathcal{T}$ is an $r$-tree, and $E^\prime\in E(\mathcal{T})$. Then for any $S\in\binom{E^\prime}{r-1}$ and $v\not\in V(\mathcal{T})$, the tree defined with edge set $E(\mathcal{T})\cup\{S\cup\{v\}\}$ is a tight $r$-tree. 

Using this definition, Kalai (see, e.g., \cite{FF}) proposed the following generalization of the Erd\H{o}s-S\'os Conjecture.

\begin{conj}[Kalai, 1984] Let $r\geq 2$, and let $\mathcal{T}$ be a tight $r$-tree on $\ell$ vertices. Then for $n$ sufficiently large,
\[\exr{n;\mathcal{T}}\,\leq\,\frac{\ell-r}{r}\binom{n}{k-1}.\]
\end{conj}

This conjecture remains open in virtually all cases. For general $r$, it is proven only in the case of trees containing an edge intersecting every other edge in $k-1$ vertices; thus the tree is essentially a star \cite{FF}.

\begin{probl} What are the extremal numbers for hyperforests containing non-path components?\end{probl}

This seems to be a quite difficult problem. There are extraordinarily few examples of even individual non-path hypertrees for which the extremal numbers are known (see, e.g., \cite{BushawThesis}, \cite{Keevash} for a survey).

\bibliographystyle{amsplain}
\bibliography{HypergraphTuran2_1}

\end{document}